\documentclass[12pt,reqno]{amsart}
\usepackage{eurosym}
\usepackage{amsfonts}
\usepackage{amsmath}
\usepackage[bookmarksnumbered, colorlinks, plainpages]{hyperref}
\usepackage{graphicx}
\usepackage{epsfig}
\usepackage{epstopdf}
\usepackage{color}
\usepackage{caption}
\usepackage{subcaption}

\hypersetup{colorlinks=true,linkcolor=red, anchorcolor=green, citecolor=cyan, urlcolor=red, filecolor=magenta, pdftoolbar=true}
\newtheorem{theorem}{Theorem}[section]

\newtheorem{corollary}[theorem]{Corollary}
\theoremstyle{definition}
\newtheorem{definition}[theorem]{Definition}
\newtheorem{example}[theorem]{Example}

\theoremstyle{remark}
\newtheorem{remark}[theorem]{Remark}
\numberwithin{equation}{section}

\begin{document}
\date{{\scriptsize Received: , Accepted: .}}
\title[Fixed-disc results via simulation functions]{Fixed-disc results via
simulation functions }

\begin{abstract}
In this paper, our aim is to obtain new fixed-disc results on metric spaces.
To do this, we present a new approach using the set of simulation functions
and some known fixed-point techniques. We do not need to have some strong conditions such as completeness or compactness of the metric
space or continuity of the self-mapping in our results. Taking only one geometric condition, we ensure
the existence of a fixed disc of a new type contractive mapping.
\end{abstract}

\author[N. Y. \"{O}ZG\"{U}R]{N\.{I}HAL YILMAZ \"{O}ZG\"{U}R}
\address[Nihal Y{\i}lmaz \"{O}zg\"{u}r]{Bal\i kesir University, Department
of Mathematics, 10145 Bal\i kesir, TURKEY}
\email{nihal@balikesir.edu.tr}
\maketitle


\setcounter{page}{1}


\section{\textbf{Introduction and preliminaries}}

\label{sec:intro}

Let $(X,d)$ be a metric space and $T$ a self-mapping on $X$. If $T$ has more
than one fixed point then the investigation of the geometric properties of
fixed points appears a natural and interesting problem. For example, let $X=%
\mathbb{R}$ be the set of all real numbers with the usual metric $%
d(x,y)=\left\vert x-y\right\vert $ for all $x,y\in \mathbb{R}$. The
self-mapping $T:\mathbb{R\rightarrow R}$ defined by $Tx=x^{2}-2$ has two
fixed points $x_{1}=-1$ and $x_{2}=2$. Fixed points of $T$ form the circle $%
C_{\frac{1}{2},\frac{3}{2}}=\left\{ x\in \mathbb{R}:\left\vert x-\frac{1}{2}%
\right\vert =\frac{3}{2}\right\} $. In recent years, the fixed-circle
problem and the fixed-disc problem have been studied with this perspective
on metric and some generalized metric spaces (see \cite{Aydi, Mlaiki, Mlaiki-Axioms, Ozgur-Tas-malaysian, Ozgur-Tas-circle-thesis, Ozgur-Tas-Celik, ozgur-aip, ozgur-discontinuity, Pant-Ozgur-Tas, Pant simulation, Pant-Nihal, Tas math, Tas, Tas-fbed, Tas-bap,
Tas-Nb, Tas-fixed-tx, Tas-aydi} for more details). As
a consequence of some fixed-circle theorems, fixed-disc results have been
also appeared. For example, the self-mapping $S$ on $\mathbb{R}$ defined by%
\begin{equation*}
Sx=\left\{
\begin{array}{ccc}
x & ; & x\in \left[ 0,2\right] \\
x+\sqrt{2} & ; & \text{otherwise}%
\end{array}%
\right.
\end{equation*}%
fixes all points of the disc $D_{1,1}=\left\{ x\in \mathbb{R}:\left\vert
x-1\right\vert \leq 1\right\} $. Clearly, $S$ fixes all circles contained in
the disc $D_{1,1}$. Therefore it is an attractive problem to study new
fixed-disc results and their consequences on metric spaces.

In this paper, our aim is to present new fixed-disc results. To do this, we
provide a new technique using simulation functions defined in \cite%
{Khojasteh}. The function $\zeta :[0,\infty )^{2}\rightarrow
\mathbb{R}
$ is said to be a simulation function, if it satisfies the following
conditions :

$(\zeta _{1})$ $\zeta (0,0)=0,$

$(\zeta _{2})$ $\zeta (t,s)<s-t$ for all $s,t>0$,

$(\zeta _{3})$ If $\{t_{n}\}$, $\{s_{n}\}$ are sequences in $(0,\infty )$
such that%
\begin{equation*}
\underset{n\rightarrow \infty }{\lim }t_{n}=\underset{n\rightarrow \infty }{%
\lim }s_{n}>0\text{,}
\end{equation*}%
then%
\begin{equation*}
\underset{n\rightarrow \infty }{\lim \sup }\zeta (t_{n},s_{n})<0\text{.}
\end{equation*}%
The set of all simulation functions is denoted by $\mathcal{Z}$ \cite%
{Khojasteh}. In \cite{Khojasteh}, the notion of a $\mathcal{Z}$-contraction
was defined to generalize the Banach contraction as follows:

\begin{definition}
\label{z-contraction}\cite{Khojasteh} Let $(X,d)$ be a metric space and $%
T:X\rightarrow X$ a mapping and $\zeta \in \mathcal{Z}$. Then $T$ is called
a $\mathcal{Z}$-contraction with respect to $\zeta $ if the following
condition is satisfied for all $x,y\in X:$
\begin{equation}
\zeta \left( d\left( Tx,Ty\right) ,d\left( x,y\right) \right) \geq 0.\text{ }
\label{definition of z-contraction}
\end{equation}
\end{definition}

Every $\mathcal{Z}$-contraction mapping is contractive and hence it is
continuous (see \cite{Chanda}, \cite{Khojasteh}, \cite{Radenovic 2017} for
basic properties and some examples of a $\mathcal{Z}$-contraction). In \cite%
{Khojasteh}, Khojasteh et al. used the notion of a simulation function to
unify several existing fixed-point results in the literature.

We note that the notion of a simulation function has many interesting
applications (see \cite{Chanda}, \cite{Fe}, \cite{Karapinar} and the
references therein). In a very recent paper, it is given a new solution to
an open problem raised by Rhoades about the discontinuity problem at fixed
point using the family of simulation functions (see \cite{Pant simulation}
and \cite{Rhoades}).

\section{\textbf{Main results}}

\label{sec:1}

Let $(X,d)$ be a metric space, $D_{x_{0},r}=\left\{ x\in
X:d(x,x_{0})\leq r\right\} $ $(r\in
\mathbb{R}
^{+}\cup \{0\})$ a disc and $T$ a self-mapping on $X$. If $Tx=x$ for all $%
x\in D_{x_{0},r}$ then the disc $D_{x_{0},r}$ is called as the fixed disc of
$T$ \cite{Tas-aydi}.

From now on we assume that $(X,d)$ is a metric space and $T:X\rightarrow X$
a self-mapping. To obtain new fixed-disc results, we define several new
contractive mappings. At first, we give the following definition.

\begin{definition}
\label{def1} Let $\zeta \in \mathcal{Z}$ be any simulation function. $T$ is
said to be a $\mathcal{Z}_{c}$-contraction with respect to $\zeta $ if there
exists an $x_{0}\in X$ such that the following condition holds for all $x\in
X:$%
\begin{equation*}
d(Tx,x)>0\Rightarrow \zeta \left( d(Tx,x),d(Tx,x_{0})\right) \geq 0.
\end{equation*}
\end{definition}

If $T$ is a $\mathcal{Z}_{c}$-contraction with respect to $\zeta $, then we
have%
\begin{equation}
d(Tx,x)<d(Tx,x_{0}),  \label{eqn1}
\end{equation}%
for all $x\in X$ with $Tx\neq x_{0}$. Indeed, if $Tx=x$ then the inequality (%
\ref{eqn1}) is satisfied trivially. If $Tx\neq x$ then $d(Tx,x)>0$. By the
definition of a $\mathcal{Z}_{c}$-contraction and the condition $(\zeta
_{2}) $, we obtain%
\begin{equation*}
0\leq \zeta \left( d(Tx,x),d(Tx,x_{0})\right) <d(Tx,x_{0})-d(Tx,x)
\end{equation*}%
and so the equation (\ref{eqn1}) is satisfied.

In all of our fixed disc results we use the number $\rho \in
\mathbb{R}
^{+}\cup \{0\}$ defined by%
\begin{equation}
\rho =\inf_{x\in X}\{d(x,Tx)\mid Tx\neq x\}\text{.}
\label{definition of radius}
\end{equation}%
We begin with the following theorem.

\begin{theorem}
\label{thm1} If $T$ is a $\mathcal{Z}_{c}$-contraction with respect to $%
\zeta $ with $x_{0}\in X$ and the condition $0<d(Tx,x_{0})\leq \rho $ holds
for all $x\in D_{x_{0},\rho }-\left\{ x_{0}\right\} $ then $D_{x_{0},\rho }$
is a fixed disc of $T$.
\end{theorem}

\begin{proof}
Let $\rho =0$. In this case we have $D_{x_{0},\rho }=\{x_{0}\}$. If $%
Tx_{0}\neq x_{0}$ then $d(x_{0},Tx_{0})>0$ and using the definition of a $%
\mathcal{Z}_{c}$-contraction we get
\begin{equation*}
\zeta \left( d(Tx_{0},x_{0}),d(Tx_{0},x_{0})\right) \geq 0.
\end{equation*}%
This is a contradiction by the condition $(\zeta _{2})$. Hence it should be $%
Tx_{0}=x_{0}$.

Assume that $\rho \neq 0$. Let $x\in D_{x_{0},\rho }$ be such that $Tx\neq x$%
. By the definition of $\rho $, we have $0<\rho \leq d(x,Tx)$ and using the
condition $(\zeta _{2})$ we find
\begin{eqnarray*}
\zeta \left( d(Tx,x),d(Tx,x_{0})\right) &<&d(Tx,x_{0})-d(Tx,x) \\
&<&\rho -d(Tx,x)\leq \rho -\rho =0,
\end{eqnarray*}%
a contradiction with the $\mathcal{Z}_{c}$-contractive property of $T$. It
should be $Tx=x$ and so, $T$ fixes the disc $D_{x_{0},\rho }$.
\end{proof}

In the following corollaries we obtain new fixed-disc results.

\begin{corollary}
\label{cor1} Let $x_{0}\in X$. If $T$ satisfies the following conditions
then $D_{x_{0},\rho }$ is a fixed disc of $T:$

$1)$ $d(Tx,x)\leq \lambda d(Tx,x_{0})$ for all $x\in X$,\newline
where $\lambda \in \left[ 0,1\right) $.

$2)$ $0<d(Tx,x_{0})\leq \rho $ holds for all $x\in D_{x_{0},\rho }-\left\{
x_{0}\right\} $.
\end{corollary}

\begin{proof}
Let us consider the function $\zeta _{1}:[0,\infty )\times \lbrack 0,\infty
)\rightarrow
\mathbb{R}
$ defined by%
\begin{equation*}
\zeta _{1}(t,s)=\lambda s-t\text{ for all }s,t\in \lbrack 0,\infty )
\end{equation*}%
(see Corollary 2.10 given in \cite{Khojasteh}). Using the hypothesis, it is
easy to see that the self-mapping $T$ is a $\mathcal{Z}_{c}$-contraction
with respect to $\zeta _{1}$ with $x_{0}\in X$. Hence the proof follows by
setting $\zeta =\zeta _{1}$ in Theorem \ref{thm1}.
\end{proof}

\begin{corollary}
\label{cor2} Let $x_{0}\in X$. If $T$ satisfies the following conditions
then $D_{x_{0},\rho }$ is a fixed disc of $T:$

$1)$ $d(Tx,x)\leq d(Tx,x_{0})-\varphi \left( d(Tx,x_{0})\right) $ for all $%
x\in X,$\newline
where $\varphi :[0,\infty )\rightarrow \lbrack 0,\infty )$ is lower semi
continuous function and $\varphi ^{-1}(0)=0$.

$2)$ $0<d(Tx,x_{0})\leq \rho $ holds for all $x\in D_{x_{0},\rho }-\left\{
x_{0}\right\} .$
\end{corollary}

\begin{proof}
Consider the function $\zeta _{2}:[0,\infty )\times \lbrack 0,\infty
)\rightarrow
\mathbb{R}
$ defined by%
\begin{equation*}
\zeta _{2}(t,s)=s-\varphi \left( s\right) -t\text{,}
\end{equation*}%
for all $s,t\in \lbrack 0,\infty )$ (see Corollary 2.11 given in \cite%
{Khojasteh}). Using the hypothesis, it is easy to verify that the
self-mapping $T$ is a $\mathcal{Z}_{c}$-contraction with respect to $\zeta
_{2}$ with $x_{0}\in X$. Hence the proof follows by setting $\zeta =\zeta
_{2}$ in Theorem \ref{thm1}.
\end{proof}

\begin{corollary}
\label{cor3} Let $x_{0}\in X$. If $T$ satisfies the following conditions
then $D_{x_{0},\rho }$ is a fixed disc of $T:$

$1)$ $d(Tx,x)\leq \varphi \left( d(Tx,x_{0})\right) d(Tx,x_{0})$ for all $%
x\in X,$\newline
where $\varphi :[0,\infty )\rightarrow \lbrack 0,1)$ be a mapping such that $%
\underset{t\rightarrow r^{+}}{\lim \sup }\varphi (t)<1$, for all $r>0$.

$2)$ $0<d(Tx,x_{0})\leq \rho $ holds for all $x\in D_{x_{0},\rho }-\left\{
x_{0}\right\} .$
\end{corollary}

\begin{proof}
Consider the function $\zeta _{3}:[0,\infty )\times \lbrack 0,\infty
)\rightarrow
\mathbb{R}
$ defined by%
\begin{equation*}
\zeta _{3}(t,s)=s\varphi \left( s\right) -t\text{,}
\end{equation*}%
for all $s,t\in \lbrack 0,\infty )$ (see Corollary 2.13 given in \cite%
{Khojasteh}). Using the hypothesis, it is easy to verify that the
self-mapping $T$ is a $\mathcal{Z}_{c}$-contraction with respect to $\zeta
_{3}$ with $x_{0}\in X$. Therefore the proof follows by setting $\zeta
=\zeta _{3}$ in Theorem \ref{thm1}.
\end{proof}

\begin{corollary}
\label{cor4} Let $x_{0}\in X$. If $T$ satisfies the following conditions
then $D_{x_{0},\rho }$ is a fixed disc of $T:$

$1)$ $d(Tx,x)\leq \eta \left( d(Tx,x_{0})\right) $ for all $x\in X,$ \newline
where $\eta :[0,\infty )\rightarrow \lbrack 0,\infty )$ be an upper semi
continuous mapping such that $\eta (t)<t$ for all $t>0$.

$2)$ $0<d(Tx,x_{0})\leq \rho $ holds for all $x\in D_{x_{0},\rho }-\left\{
x_{0}\right\} .$
\end{corollary}

\begin{proof}
Consider the function $\zeta _{4}:[0,\infty )\times \lbrack 0,\infty
)\rightarrow
\mathbb{R}
$ defined by%
\begin{equation*}
\zeta _{4}(t,s)=\eta \left( s\right) -t\text{,}
\end{equation*}%
for all $s,t\in \lbrack 0,\infty )$ (see Corollary 2.14 given in \cite%
{Khojasteh}). Using the hypothesis, it is easy to verify that the
self-mapping $T$ is a $\mathcal{Z}_{c}$-contraction with respect to $\zeta
_{4}$ with $x_{0}\in X$. Therefore the proof follows by setting $\zeta
=\zeta _{4}$ in Theorem \ref{thm1}.
\end{proof}

\begin{corollary}
\label{cor5} Let $x_{0}\in X$. If $T$ satisfies the following conditions
then $D_{x_{0},\rho }$ is a fixed disc of $T:$

$1)$ $\int\limits_{0}^{d(Tx,x)}\phi (t)dt\leq d(Tx,x_{0})$ for all $x\in X,$%
\newline
where $\phi :[0,\infty )\rightarrow \lbrack 0,\infty )$ is a function such
that $\int\limits_{0}^{\varepsilon }\phi (t)dt$ exists and $%
\int\limits_{0}^{\varepsilon }\phi (t)dt>\varepsilon $, for each $%
\varepsilon >0$.

$2)$ $0<d(Tx,x_{0})\leq \rho $ holds for all $x\in D_{x_{0},\rho }-\left\{
x_{0}\right\} .$
\end{corollary}

\begin{proof}
Consider the function $\zeta _{5}:[0,\infty )\times \lbrack 0,\infty
)\rightarrow
\mathbb{R}
$ defined by%
\begin{equation*}
\zeta _{5}(t,s)=s-\int\limits_{0}^{t}\phi (u)du\text{,}
\end{equation*}%
for all $s,t\in \lbrack 0,\infty )$ (see Corollary 2.15 given in \cite%
{Khojasteh}). Using the hypothesis, it is easy to verify that the
self-mapping $T$ is a $\mathcal{Z}_{c}$-contraction with respect to $\zeta
_{5}$ with $x_{0}\in X$. Therefore the proof follows by taking $\zeta =\zeta
_{4}$ in Theorem \ref{thm1}.
\end{proof}

We give the following example.

\begin{example}
\label{exm1} Let $X=%
\mathbb{R}
$ and $\left( X,d\right) $ be the usual metric space with $d(x,y)=\left\vert
x-y\right\vert $. Let us define the self-mapping $T_{1}:X\rightarrow X$ as%
\begin{equation*}
T_{1}x=\left\{
\begin{array}{ccc}
x & ; & x\in \left[ -1,1\right] \\
2x & ; & x\in \left( -\infty ,-1\right) \cup \left( 1,\infty \right)%
\end{array}%
\right. \text{,}
\end{equation*}%
for all $x\in
\mathbb{R}
$. Then $T_{1}$ is a $\mathcal{Z}_{c}$-contraction with $\rho =1$, $x_{0}=0$
and the function $\zeta _{6}:[0,\infty )^{2}\rightarrow
\mathbb{R}
$ defined as $\zeta _{6}(t,s)=\frac{3}{4}s-t$. Indeed, it is clear that
\begin{equation*}
0<d(T_{1}x,0)=\left\vert x-0\right\vert =\left\vert x\right\vert \leq 1\text{,}
\end{equation*}%
for all $x\in D_{0,1}-\left\{ 0\right\} $ and we have
\begin{equation*}
\zeta _{6}\left( d(T_{1}x,x),d(T_{1}x,x_{0})\right) =\zeta \left( \left\vert
x\right\vert ,\left\vert 2x\right\vert \right) =\frac{1}{2}\left\vert
x\right\vert >0
\end{equation*}%
for all $x\in \mathbb{R}$ such that $d(Tx,x)>0$. Consequently, $T_{1}$ fixes
the disc $D_{0,1}=[-1,1]$.

Now we consider the self-mapping $T_{2}:X\rightarrow X$ defined by%
\begin{equation*}
T_{2}x=\left\{
\begin{array}{ccc}
x & ; & \left\vert x-x_{0}\right\vert \leq \mu \\
2x_{0} & ; & \left\vert x-x_{0}\right\vert >\mu%
\end{array}%
\right. \text{,}
\end{equation*}%
for all $x\in \mathbb{R}$ with $0<x_{0}$ and $\mu \geq 2x_{0}$. The
self-mapping $T_{2}$ is not a $\mathcal{Z}_{c}$-contraction with respect to
any $\zeta \in \mathcal{Z}$ with $x_{0}\in X$. But $T_{2}$ fixes the disc $%
D_{x_{0},\mu }$. Indeed, by the condition $(\zeta _{2})$, for all $x\in
\left( -\infty ,x_{0}-\mu \right) \cup \left( x_{0}+\mu ,\infty \right) $ we
have
\begin{eqnarray*}
\zeta \left( d(Tx,x),d(Tx,x_{0})\right) &=&\zeta \left( \left\vert
2x_{0}-x\right\vert ,\left\vert 2x_{0}-x_{0}\right\vert \right) \\
&=&\zeta \left( \left\vert 2x_{0}-x\right\vert ,\left\vert x_{0}\right\vert
\right) <\left\vert x_{0}\right\vert -\left\vert 2x_{0}-x\right\vert <0\text{%
.}
\end{eqnarray*}%
This example shows that the converse statement of Theorem \ref{thm1} is not
true everywhen.
\end{example}

\begin{remark}
$1)$ We note that the radius $\rho $ of the fixed disc $D_{x_{0},\rho }$ is
not maximal in Theorem \ref{thm1} $($resp. Corollary \ref{cor1}-Corollary %
\ref{cor5}$)$. That is, if $D_{x_{0},\rho _{1}}$ is another fixed disc of
the self-mapping $T$ then it can be $\rho \leq \rho _{1}$. Indeed, if we
consider the self mapping $T_{3}:%
\mathbb{R}
\rightarrow
\mathbb{R}
$ defined by%
\begin{equation*}
T_{3}x=\left\{
\begin{array}{ccc}
x & ; & x\in \left[ -3,3\right] \\
x+1 & ; & \text{otherwise}%
\end{array}%
\right.
\end{equation*}%
with the usual metric on $%
\mathbb{R}
$, then the self-mapping $T_{3}$ is a $\mathcal{Z}_{c}$-contraction with $%
\rho =1$, $x_{0}=0$ and the function $\zeta _{7}:[0,\infty )^{2}\rightarrow
\mathbb{R}
$ defined as $\zeta _{7}(t,s)=\frac{1}{2}s-t$. Hence, $T_{1}$ fixes the disc
$D_{0,1}=[-1,1]$ by Theorem \ref{thm1}. But the disc $D_{0,2}=[-2,2]$ is
another fixed disc of the self-mapping $T_{3}$.

$2)$ The radius $\rho $ of the fixed disc $D_{x_{0},\rho }$ is independent
from the center $x_{0}$ in Theorem \ref{thm1} $($resp. Corollary \ref{cor1}%
-Corollary \ref{cor5}$)$. Again, if we consider the self-mapping $T_{3}$
defined in $(1)$, it is easy to verify that $T_{3}$ is also a $\mathcal{Z}%
_{c}$-contraction with $\rho =1$, $x_{0}=1$ and the function $\zeta _{7}$.
Clearly, the disc $D_{1,1}=[0,2]$ is another fixed disc of $T_{3}$.
\end{remark}

In \cite{Aydi}, Aydi et al. introduced the notion of a $\alpha $-$x_{0}$%
-admissible map as follows:

\begin{definition}
\label{def2}\cite{Aydi} Let $X$ be a non-empty set. Given a function $\alpha
:X\times X\rightarrow (0,\infty )$ and $x_{0}\in X.$ $T$ is said to be an $%
\alpha $-$x_{0}$-admissible map if for every $x\in X,$
\begin{equation*}
\alpha (x_{0},x)\geq 1\Rightarrow \alpha (x_{0},Tx)\geq 1.
\end{equation*}
\end{definition}

Then using this notion it was given new fixed-disc results on a rectangular
metric space in \cite{Aydi}. Now we give the following definition.

\begin{definition}
\label{def3} Let $T$ be a self-mapping defined on a metric space $(X,d)$. If
there exist $\zeta \in \mathcal{Z}$, $x_{0}\in X$ and $\alpha :X\times
X\rightarrow (0,\infty )$ such that
\begin{equation*}
d(Tx,x)>0\Rightarrow \zeta \left( \alpha
(x_{0},Tx)d(x,Tx),d(Tx,x_{0})\right) \geq 0\text{ for all }x\in X,
\end{equation*}%
then $T$ is called as an $\alpha $-$\mathcal{Z}_{c}$-contraction with
respect to $\zeta $.
\end{definition}

\begin{remark}
$1)$ If $T$ is an $\alpha $-$\mathcal{Z}_{c}$-contraction with respect to $%
\zeta $, then we have%
\begin{equation}
\alpha (x_{0},Tx)d(x,Tx)<d(Tx,x_{0})\text{,}  \label{eqn2}
\end{equation}%
for all $x\in X$ such that $Tx\neq x_{0}$. If $Tx\neq x_{0}$ then we have $%
d(Tx,x_{0})>0$.

Case 1. If $Tx=x$, then $\alpha (x_{0},Tx)d(x,Tx)=0<d(Tx,x_{0})$.

Case 2. If $Tx\neq x$, then $d(Tx,x)>0$. Since $\alpha (x_{0},Tx)>0$, then
by the condition $(\zeta _{2})$ and the definition of an $\alpha $-$\mathcal{%
Z}_{c}$-contraction, we find
\begin{equation*}
0\leq \zeta \left( \alpha (x_{0},Tx)d(x,Tx),d(Tx,x_{0})\right)
<d(Tx,x_{0})-\alpha (x_{0},Tx)d(x,Tx)
\end{equation*}%
and hence%
\begin{equation*}
\alpha (x_{0},Tx)d(x,Tx)<d(Tx,x_{0}).
\end{equation*}

$2)$ If $\alpha (x_{0},Tx)=1$ then an $\alpha $-$\mathcal{Z}_{c}$%
-contraction $T$ turns into a $\mathcal{Z}_{c}$-contraction with respect to $%
\zeta $ and the equation $($\ref{eqn2}$)$ turns in to the equation $($\ref%
{eqn1}$)$.
\end{remark}

Now we give the following theorem.

\begin{theorem}
\label{thm2} Let $T$ be an $\alpha $-$\mathcal{Z}_{c}$-contraction with
respect to $\zeta $ with $x_{0}\in X$. Assume that $T$ is $\alpha $-$x_{0}$%
-admissible. If $\alpha (x_{0},x)\geq 1$ for $x\in D_{x_{0},\rho }$ and $%
0<d(Tx,x_{0})\leq \rho $ for $x\in D_{x_{0},\rho }-\left\{ x_{0}\right\} $,
then $D_{x_{0},\rho }$ is a fixed disc of $T$.
\end{theorem}

\begin{proof}
Let $\rho =0$. In this case $D_{x_{0},\rho }=\{x_{0}\}$ and the $\alpha $-$%
\mathcal{Z}_{c}$-contractive hypothesis yields $Tx_{0}=x_{0}$. Indeed, if $%
Tx_{0}\neq x_{0}$ then $d(x_{0},Tx_{0})>0$ and using the definition of an $%
\alpha $- $\mathcal{Z}_{c}$-contraction we get
\begin{equation*}
\zeta \left( \alpha (x_{0},Tx_{0})d(Tx_{0},x_{0}),d(Tx_{0},x_{0})\right)
\geq 0.
\end{equation*}%
We have a contradiction by the condition $(\zeta _{2})$. Hence it should be $%
Tx_{0}=x_{0}$.

Assume that $\rho \neq 0$. Let $x\in D_{x_{0},\rho }$ be such that $Tx\neq x$%
. By the hypothesis, we have $\alpha (x_{0},x)\geq 1$ and by the $\alpha $-$%
x_{0}$-admissible property of $T$ we get $\alpha (x_{0},Tx)\geq 1$. Then
using the condition $(\zeta _{2})$ we find
\begin{eqnarray*}
\zeta \left( \alpha (x_{0},Tx)d(Tx,x),d(Tx,x_{0})\right)
&<&d(Tx,x_{0})-\alpha (x_{0},Tx)d(Tx,x) \\
&<&\rho -d(Tx,x)\leq \rho -\rho =0,
\end{eqnarray*}%
a contradiction with the $\alpha $-$\mathcal{Z}_{c}$-contractive property of
$T$. It should be $Tx=x$ and so, $T$ fixes the disc $D_{x_{0},\rho }$.
\end{proof}

Let us consider the number $m^{\ast }(x,y)$ defined as follows:%
\begin{equation}
m^{\ast }(x,y)=\max \left\{ d(x,y),d(x,Tx),d(y,Ty),\frac{d(x,Ty)+d(y,Tx)}{2}%
\right\} \text{.}  \label{definition of the number m}
\end{equation}%
Using simulation functions and the number $m^{\ast }(x,y)$, new fixed-point
results were obtained in \cite{Padcharoen}. Also, using this number, some
discontinuity results at fixed point was given in \cite{Bisht-2017-1}. Now
we obtain a new fixed-disc result using the number $m^{\ast }(x,y)$ and the
set of simulation functions.

We give the following definition.

\begin{definition}
\label{def4} Let $(X,d)$ be a metric space, $T:X\rightarrow X$ a
self-mapping and $\zeta \in \mathcal{Z}$. $T$ is said to be a \'{C}iri\'{c}
type $\mathcal{Z}_{c}$-contraction with respect to $\zeta $ if there exist
an $x_{0}\in X$ such that the following condition holds for all $x\in X:$%
\begin{equation*}
d(Tx,x)>0\Rightarrow \zeta \left( d(Tx,x),m^{\ast }(x,x_{0})\right) \geq 0.
\end{equation*}%
Now we give the following theorem.
\end{definition}

\begin{theorem}
\label{thm3} Let $(X,d)$ be a metric space and $T:X\rightarrow X$ a \'{C}iri%
\'{c} type $\mathcal{Z}_{c}$- contraction with respect to $\zeta $ with $%
x_{0}\in X$. If the condition $0<d(Tx,x_{0})\leq \rho $ holds for all $x\in
D_{x_{0},\rho }-\left\{ x_{0}\right\} $ then $D_{x_{0},\rho }$ is a fixed
disc of $T$.
\end{theorem}

\begin{proof}
Let $\rho =0$. In this case we have $D_{x_{0},\rho }=\{x_{0}\}$ and the \'{C}%
iri\'{c} type $\mathcal{Z}_{c}$-contractive hypothesis yields $Tx_{0}=x_{0}$%
. Indeed, if $Tx_{0}\neq x_{0}$ then we have $d(x_{0},Tx_{0})>0$. By the
definition of a \'{C}iri\'{c} type $\mathcal{Z}_{c}$-contraction we have
\begin{equation}
\zeta \left( d(Tx_{0},x_{0}),m^{\ast }(x_{0},x_{0})\right) \geq 0.
\label{eqn3}
\end{equation}%
Since we have
\begin{eqnarray*}
m^{\ast }(x_{0},x_{0}) &=&\max \left\{
d(x_{0},x_{0}),d(x_{0},Tx_{0}),d(x_{0},Tx_{0}),\frac{%
d(x_{0},Tx_{0})+d(x_{0},Tx_{0})}{2}\right\} \\
&=&d(x_{0},Tx_{0}),
\end{eqnarray*}%
we find%
\begin{equation*}
\zeta \left( d(Tx_{0},x_{0}),m^{\ast }(x_{0},x_{0})\right) =\zeta \left(
d(Tx_{0},x_{0}),d(x_{0},Tx_{0})\right) <0
\end{equation*}%
by the condition $(\zeta _{2})$. This is a contradiction to the equation (%
\ref{eqn3}). Hence it should be $Tx_{0}=x_{0}$.

Assume that $\rho \neq 0$. Let $x\in D_{x_{0},\rho }$ be such that $Tx\neq x$%
. Then we have
\begin{eqnarray*}
m^{\ast }(x,x_{0}) &=&\max \left\{ d(x,x_{0}),d(x,Tx),d(x_{0},Tx_{0}),\frac{%
d(x,Tx_{0})+d(x_{0},Tx)}{2}\right\} \\
&=&\max \left\{ d(x,x_{0}),d(x,Tx),\frac{d(x,x_{0})+d(x_{0},Tx)}{2}\right\}
\text{.}
\end{eqnarray*}%
By the hypothesis, we have%
\begin{equation*}
\zeta \left( d(Tx,x),m^{\ast }(x,x_{0})\right) \geq 0
\end{equation*}%
and so
\begin{equation}
\zeta \left( d(Tx,x),\max \left\{ d(x,x_{0}),d(x,Tx),\frac{%
d(x,x_{0})+d(x_{0},Tx)}{2}\right\} \right) \geq 0\text{.}  \label{eqn4}
\end{equation}%
We have the following cases:

Case 1. Let $\max \left\{ d(x,x_{0}),d(x,Tx),\frac{d(x,Tx_{0})+d(x_{0},Tx)}{2%
}\right\} =d(x,x_{0})$. From (\ref{eqn4}) we get
\begin{equation*}
\zeta \left( d(Tx,x),d(x,x_{0})\right) \geq 0.
\end{equation*}%
Using the condition $(\zeta _{2})$ and considering definition of $\rho $, we
find%
\begin{equation*}
\zeta \left( d(Tx,x),d(x,x_{0})\right) <d(x,x_{0})-d(Tx,x)<\rho
-d(Tx,x)<\rho -\rho =0.
\end{equation*}%
This is a contradiction with the \'{C}iri\'{c} type $\mathcal{Z}_{c}$%
-contractive property of $T$.

Case 2. Let $\max \left\{ d(x,x_{0}),d(x,Tx),\frac{d(x,x_{0})+d(x_{0},Tx)}{2}%
\right\} =d(x,Tx)$. From (\ref{eqn4}) we get
\begin{equation*}
\zeta \left( d(Tx,x),d(x,Tx)\right) \geq 0.
\end{equation*}%
Using the condition $(\zeta _{2})$, again we get a contradiction.

Case 3. Let $\max \left\{ d(x,x_{0}),d(x,Tx),\frac{d(x,x_{0})+d(x_{0},Tx)}{2}%
\right\} =\frac{d(x,x_{0})+d(x_{0},Tx)}{2}$. From (\ref{eqn4}) we get
\begin{equation*}
\zeta \left( d(Tx,x),\frac{d(x,x_{0})+d(x_{0},Tx)}{2}\right) \geq 0.
\end{equation*}%
Using the condition $(\zeta _{2})$, we get%
\begin{eqnarray*}
\zeta \left( d(Tx,x),\frac{d(x,x_{0})+d(x_{0},Tx)}{2}\right) &<&\frac{%
d(x,x_{0})+d(x_{0},Tx)}{2}-d(Tx,x) \\
&<&\rho -d(Tx,x)<\rho -\rho =0.
\end{eqnarray*}%
Again this is a contradiction with the \'{C}iri\'{c} type $\mathcal{Z}_{c}$%
-contractive property of $T$.

In all of the above cases we have a contradiction. Hence it should be $Tx=x$
and consequently, $T$ fixes the disc $D_{x_{0},\rho }$.
\end{proof}

\section{\textbf{A common fixed-disc theorem}}

\label{sec:2} In this section, we give a common fixed-disc result for a pair
of self-mappings $\left( T,S\right) $ of a metric space $\left( X,d\right) $%
. If $Tx=Sx=x$ for all $x\in D_{x_{0},r}$ then the disc $D_{x_{0},r}$ is
called as the common fixed disc of the pair $\left( T,S\right) $. At first,
we modify the number defined in (\ref{definition of the number m}) for a
pair of self-mappings as follows:%
\begin{equation}
m_{S,T}^{\ast }(x,y)=\max \left\{ d(Tx,Sy),d(Tx,Sx),d(Ty,Sy),\frac{%
d(Tx,Sy)+d(Ty,Sx)}{2}\right\} \text{.}  \label{common 1}
\end{equation}%
Then we give the following theorem using the numbers $m_{S,T}^{\ast }(x,y)$, $\rho $, $r$ $\in
\mathbb{R}
^{+}\cup \{0\}$ defined by%
\begin{equation}
r =\inf_{x\in X}\{d(Tx,Sx)\mid Tx\neq Sx\}
\label{definition of mu}
\end{equation}

and

\begin{equation}
\mu = \min \left\{\rho, r\right\} \text{.}
\end{equation}

\begin{theorem}
\label{thm4} Let $T,S:$ $X\rightarrow X$ be two self-mappings on a metric
space. Assume that there exists $\zeta \in \mathcal{Z}$ and $x_{0}\in X$
such that
\begin{equation*}
d(Tx,Sx)>0\Rightarrow \zeta \left( d\left( Tx,Sx\right) ,m_{S,T}^{\ast
}(x,x_{0})\right) \geq 0\text{ for all }x\in X
\end{equation*}%
and
\begin{equation*}
d(Tx,x_{0})\leq \mu \text{, }d(Sx,x_{0})\leq \mu \text{ for all }x\in
D_{x_{0},\mu }\text{.}
\end{equation*}%
If $T$ is a $\mathcal{Z}_{c}$-contraction with $0<d(Tx,x_{0})\leq \rho$ for $x\in
D_{x_{0},\rho }-\left\{ x_{0}\right\}$ $($or $S$ is a $\mathcal{Z}_{c}$%
-contraction with $0<d(Sx,x_{0})\leq \rho$ for $x\in
D_{x_{0},\rho }-\left\{ x_{0}\right\}$$)$, then $D_{x_{0},\mu }$ is a common fixed disc of $T$ and $S$ in $X$.
\end{theorem}

\begin{proof}
Let $\mu =0$. In this case we have $D_{x_{0},\mu }=\{x_{0}\}$ and by the
hypothesis, we get $Tx_{0}=Sx_{0}=x_{0}$.

Let $\mu >0$. At first, we show that $x_{0}$ is a coincidence point of $T$
and $S$, that is, $Tx_{0}=Sx_{0}$. Assume that $Tx_{0}\neq Sx_{0}$ and so $%
d(Tx_{0},Sx_{0})>0$. Then we have
\begin{equation*}
\zeta \left( d\left( Tx_{0},Sx_{0}\right) ,m_{S,T}^{\ast
}(x_{0},x_{0})\right) =\zeta \left( d(Tx_{0},Sx_{0}),d(Tx_{0},Sx_{0})\right)
\text{.}
\end{equation*}%
But this is a contradiction by the condition $(\zeta _{2})$. Hence we find $%
Tx_{0}=Sx_{0}$, that is, $x_{0}$ is a coincidence point of $T$ and $S$. If $%
T $ is a $\mathcal{Z}_{c}$-contraction $($or $S$ is a $\mathcal{Z}_{c}$%
-contraction$)$ then we have $Tx_{0}=x_{0}$ $($or $Sx_{0}=x_{0})$ and $%
Tx_{0}=Sx_{0}=x_{0}$.

Let $x\in D_{x_{0},\mu }$ be an arbitrary point. Suppose $Tx\neq Sx$ and so $%
d(Tx,Sx)>0$. Using the hypothesis $d(Tx,x_{0})\leq \mu $, $d(Sx,x_{0})\leq
\mu $ for all $x\in D_{x_{0},\mu }$ and considering the definition of $\mu $
we get
\begin{eqnarray*}
\zeta \left( d\left( Tx,Sx\right) ,m_{S,T}^{\ast }(x,x_{0})\right) &=&\zeta
\left( d\left( Tx,Sx\right) ,\max \left\{
\begin{array}{c}
d(Tx,Sx_{0}),d(Tx,Sx), \\
d(Tx_{0},Sx_{0}),\frac{d(Tx,Sx_{0})+d(Tx_{0},Sx)}{2}%
\end{array}%
\right\} \right) \\
&=&\zeta \left( d\left( Tx,Sx\right) ,\max \left\{
\begin{array}{c}
d(Tx,x_{0}),d(Tx,Sx), \\
0,\frac{d(Tx,x_{0})+d(x_{0},Sx)}{2}%
\end{array}%
\right\} \right) \\
&=&\zeta \left( d\left( Tx,Sx\right) ,d(Tx,Sx)\right) .
\end{eqnarray*}%
This leads a contradiction by the condition $(\zeta _{2})$. Therefore $x$ is
a coincidence point of $T$ and $S$.

Now, if $u\in D_{x_{0},\mu}$ is a fixed point of $T$ then clearly $u$ is also
a fixed point of $S$ and vice versa. If $T$ is a $\mathcal{Z}_{c}$%
-contraction $($or $S$ is a $\mathcal{Z}_{c}$-contraction$)$ then by Theorem %
\ref{thm1}, we have $Tx=x$ (or $Sx=x$) and hence $Tx=Sx=x$ for all $x\in
D_{x_{0},\mu}$. That is, the disc $D_{x_{0},\mu}$ is a common fixed-disc of $T$
and $S$.
\end{proof}

\begin{example}
Let us consider the usual metric space $X=%
\mathbb{R}
$ and the self-mapping $T_{1}$ defined in Example \ref{exm1}. Define the
self-mapping $T_{4}:%
\mathbb{R}
\rightarrow
\mathbb{R}
$ by%
\begin{equation*}
T_{4}x=\left\{
\begin{array}{ccc}
x & ; & x\in \left[ -3,3\right] \\
3x & ; & x\in \left( -\infty ,-3\right) \cup \left( 3,\infty \right)%
\end{array}%
\right. .
\end{equation*}%
Clearly, we have $\mu =$ $1$. Then the pair $\left( T_{1},T_{4}\right) $
satisfies the conditions of Theorem \ref{thm4} for $\mu =$ $1$, $x_{0}=0$ and the
function $\zeta _{6}:[0,\infty )^{2}\rightarrow
\mathbb{R}
$ defined as $\zeta _{6}(t,s)=\frac{3}{4}s-t$. Hence the disc $D_{0,1}=\left[
-1,1\right] $ is the common fixed disc of the self-mappings $T_{1}$ and $%
T_{4}$.
\end{example}

\section{\textbf{Conclusion and future work}}

In this paper, we have obtained new fixed-disc results presenting a new
approach via simulation functions. Using similar approaches, it can be
studied new fixed-disc results on metric and some generalized metric spaces.
As a future work, it is a meaningful problem to investigate some conditions
to exclude the identity map of $X$ from Theorem \ref{thm1}, Theorem \ref%
{thm2}, Theorem \ref{thm3} and related results. On the other hand, it is
worth to mention that most of the popular activation functions used in
neural networks are those mappings having fixed-discs. For example,
exponential linear unit (ELU) function defined by%
\begin{equation*}
f(x)=\left\{
\begin{array}{ccc}
x & ; & \text{if }x\geq 0 \\
\alpha (\exp (x)-1) & ; & \text{if }x<0%
\end{array}%
,\right.
\end{equation*}%
where $\alpha $ is constant of ELUs, and S-shaped rectified linear unit
function (SReLU) defined by%
\begin{equation*}
h(x_{i})=\left\{
\begin{array}{ccc}
t_{i}^{r}+a_{i}^{r}(x-t_{i}^{r}) & ; & x_{i}\geq t_{i}^{r} \\
x_{i} & ; & t_{i}^{r}>x_{i}>t_{i}^{l} \\
t_{i}^{l}+a_{i}^{l}(x-t_{i}^{l}) & ; & x_{i}\leq t_{i}^{l}%
\end{array}%
\right. ,
\end{equation*}%
where $\left\{ t_{i}^{r},a_{i}^{r},a_{i}^{l},t_{i}^{l}\right\} $ are four
learnable parameters used to model an individual SReLU activation unit, are
well-known activation functions (see \cite{Clevert} and \cite{Jin} for more
details). Therefore, it is important to study of features of mappings having
fixed-discs.



\end{document}